\documentclass[12pt]{amsart}
\usepackage{amssymb}
\usepackage{graphics}
\usepackage{latexsym}
\usepackage{amsmath}
\usepackage{amssymb,amsthm,amsfonts}
\usepackage{amscd}
\usepackage[arrow, matrix, curve]{xy}
\usepackage{syntonly}
\usepackage[small,nohug,heads=vee]{diagrams}
\diagramstyle[labelstyle=\scriptstyle] 
\title[]{Maximal families of Calabi-Yau manifolds with minimal length Yukawa coupling}
\author[Mao Sheng]{Mao Sheng}
\author[Jinxing Xu]{Jinxing Xu}
\email{msheng@ustc.edu.cn}\email{xujx02@ustc.edu.cn}
\address{School of Mathematical Sciences,
University of Science and Technology of China, Hefei, 230026, China}
\author[Kang Zuo]{Kang Zuo}
\email{zuok@uni-mainz.de}
\address{Institut f\"{u}r  Mathematik, Universit\"{a}t
Mainz, Mainz, 55099, Germany}
\begin{document}
\theoremstyle{plain}
\newtheorem{thm}{Theorem}[section]
\newtheorem{theorem}[thm]{Theorem}
\newtheorem{lemma}[thm]{Lemma}
\newtheorem{corollary}[thm]{Corollary}
\newtheorem{proposition}[thm]{Proposition}
\newtheorem{addendum}[thm]{Addendum}
\newtheorem{variant}[thm]{Variant}
\theoremstyle{definition}
\newtheorem{construction}[thm]{Construction}
\newtheorem{notations}[thm]{Notations}
\newtheorem{question}[thm]{Question}
\newtheorem{problem}[thm]{Problem}
\newtheorem{remark}[thm]{Remark}
\newtheorem{remarks}[thm]{Remarks}
\newtheorem{definition}[thm]{Definition}
\newtheorem{claim}[thm]{Claim}
\newtheorem{assumption}[thm]{Assumption}
\newtheorem{assumptions}[thm]{Assumptions}
\newtheorem{properties}[thm]{Properties}
\newtheorem{example}[thm]{Example}
\newtheorem{conjecture}[thm]{Conjecture}
\numberwithin{equation}{thm}

\newcommand{\sA}{{\mathcal A}}
\newcommand{\sB}{{\mathcal B}}
\newcommand{\sC}{{\mathcal C}}
\newcommand{\sD}{{\mathcal D}}
\newcommand{\sE}{{\mathcal E}}
\newcommand{\sF}{{\mathcal F}}
\newcommand{\sG}{{\mathcal G}}
\newcommand{\sH}{{\mathcal H}}
\newcommand{\sI}{{\mathcal I}}
\newcommand{\sJ}{{\mathcal J}}
\newcommand{\sK}{{\mathcal K}}
\newcommand{\sL}{{\mathcal L}}
\newcommand{\sM}{{\mathcal M}}
\newcommand{\sN}{{\mathcal N}}
\newcommand{\sO}{{\mathcal O}}
\newcommand{\sP}{{\mathcal P}}
\newcommand{\sQ}{{\mathcal Q}}
\newcommand{\sR}{{\mathcal R}}
\newcommand{\sS}{{\mathcal S}}
\newcommand{\sT}{{\mathcal T}}
\newcommand{\sU}{{\mathcal U}}
\newcommand{\sV}{{\mathcal V}}
\newcommand{\sW}{{\mathcal W}}
\newcommand{\sX}{{\mathcal X}}
\newcommand{\sY}{{\mathcal Y}}
\newcommand{\sZ}{{\mathcal Z}}
\newcommand{\A}{{\mathbb A}}
\newcommand{\B}{{\mathbb B}}
\newcommand{\C}{{\mathbb C}}
\newcommand{\D}{{\mathbb D}}
\newcommand{\E}{{\mathbb E}}
\newcommand{\F}{{\mathbb F}}
\newcommand{\G}{{\mathbb G}}
\newcommand{\HH}{{\mathbb H}}
\newcommand{\I}{{\mathbb I}}
\newcommand{\J}{{\mathbb J}}
\renewcommand{\L}{{\mathbb L}}
\newcommand{\M}{{\mathbb M}}
\newcommand{\N}{{\mathbb N}}
\renewcommand{\P}{{\mathbb P}}
\newcommand{\Q}{{\mathbb Q}}
\newcommand{\R}{{\mathbb R}}
\newcommand{\SSS}{{\mathbb S}}
\newcommand{\T}{{\mathbb T}}
\newcommand{\U}{{\mathbb U}}
\newcommand{\V}{{\mathbb V}}
\newcommand{\W}{{\mathbb W}}
\newcommand{\X}{{\mathbb X}}
\newcommand{\Y}{{\mathbb Y}}
\newcommand{\Z}{{\mathbb Z}}
\newcommand{\id}{{\rm id}}
\newcommand{\rank}{{\rm rank}}
\newcommand{\END}{{\mathbb E}{\rm nd}}
\newcommand{\End}{{\rm End}}
\newcommand{\Hom}{{\rm Hom}}
\newcommand{\Hg}{{\rm Hg}}
\newcommand{\tr}{{\rm tr}}
\newcommand{\Sl}{{\rm Sl}}
\newcommand{\Gl}{{\rm Gl}}
\newcommand{\Cor}{{\rm Cor}}
\newcommand{\Aut}{\mathrm{Aut}}
\newcommand{\Sym}{\mathrm{Sym}}
\newcommand{\ModuliCY}{\mathfrak{M}_{CY}}
\newcommand{\HyperCY}{\mathfrak{H}_{CY}}
\newcommand{\ModuliAR}{\mathfrak{M}_{AR}}
\newcommand{\Modulione}{\mathfrak{M}_{1,n+3}}
\newcommand{\Modulin}{\mathfrak{M}_{n,n+3}}
\newcommand{\Gal}{\mathrm{Gal}}
\newcommand{\Spec}{\mathrm{Spec}}
\newcommand{\Jac}{\mathrm{Jac}}
\newcommand{\Proj}{\mathrm{Proj}}

\thanks{This work is supported by the SFB/TR 45 `Periods, Moduli
Spaces and Arithmetic of Algebraic Varieties' of the DFG, and
partially supported by the University of Science and Technology of
China.} \maketitle

\begin{abstract}
For each natural odd number $n\geq 3$, we exhibit a maximal family
of $n$-dimensional Calabi-Yau manifolds whose Yukawa coupling length
is one. As a consequence, Shafarevich's conjecture holds true for
these families. Moreover, it follows from Deligne-Mostow \cite{DM}
and Mostow \cite{Mostow0}-\cite{Mostow} that, for $n=3$, it can be
partially compactified to a Shimura family of ball type, and for
$n=5, \ 9$, there is a sub $\Q$-PVHS of the family uniformizing a
Zariski open subset of an arithmetic ball quotient.
\end{abstract}

\section{Introduction}
The local Torelli theorem for Calabi-Yau (abbreviated as CY)
manifolds says that the Kodaira-Spencer map for a versal local
deformation of a CY manifold is an isomorphism. This important fact
has a consequence on the Yukawa coupling length which is introduced
in the work \cite{VZ}. For a family $f: \sX\to S$ of CY manifolds of
dimension $n$, let
$$
(E=\bigoplus_{p+q=n}E^{p,q},\theta=\bigoplus_{p+q=n}\theta^{p,q})
$$
be the associated Higgs bundle, where
$E^{p,q}=R^qf_*\Omega^p_{\sX/S}$ and the Higgs field
$$
\theta^{p,q}: E^{p,q}\to E^{p-1,q+1}\otimes \Omega_S
$$
is given by the cup product with the Kodaira-Spencer map. The
\emph{length} of the Yukawa coupling $\varsigma(f)$ of $f$ is then
defined by
$$
\varsigma(f)=\min\{i\geq 1, \theta^i=0\}-1,
$$
where $\theta^i$ is the $i$-th iterated Higgs field
$$
\theta^i: E^{n,0}\stackrel{\theta^{n,0}}{\longrightarrow}
E^{n-1,1}\otimes
\Omega_S\stackrel{\theta^{n-1,1}}{\longrightarrow}\cdots
\stackrel{\theta^{n-i+1,i-1}}{\longrightarrow} E^{n-i,i}\otimes
S^{i}\Omega_S.
$$
The local Torelli theorem implies that, for a non-isotrivial family
$f$ of $n$-dimensional CY manifolds, it holds true that $1\leq
\varsigma(f)\leq n$. The connection of Yukawa coupling length with
Shararevich's conjecture for CY manifolds has been intensively
studied (see e.g. \cite{LTYZ2},\cite{Zhang}). It has been shown
that, for example, if the Yukawa coupling length of $f$ is maximal,
i.e, $\varsigma(f)=n$, then $f$ is rigid. The maximality of Yukawa
coupling length seems to be very often for moduli spaces of CY
manifolds as anti-canonical classes of a toric variety (it is the
case for moduli spaces of CY manifolds with a maximal degeneration
point). Our motivation is then to look for many examples of moduli
spaces of CY manifolds whose Yukawa coupling lengths are minimal,
i.e. one. As far as we know, higher dimensional examples are rare in
the literature. What we have obtained in this paper is an infinite
series of maximal families of $n$-dimensional CY manifolds with
Yukawa coupling length one for any odd $n\geq 3$. Here a family $f$
is said to be maximal if it is locally a versal deformation of each
CY closed fiber of $f$. Our main result is summarized as follows:
\begin{theorem}
Let $n\geq 3$ be an odd number and $\Modulin$ be the moduli space of
$n+3$ hyperplane arrangements of $\P^n$ in general position. Let
$f_n: \sX_n\to \Modulin$ be the family of $\frac{n+3}{2}$-fold
cyclic covers of $\P^n$ branched along the $n+3$ hyperplanes in
general position. Then the following statements are true:
\begin{itemize}
    \item [(i)] The family $f_n$ admits a simultaneous
    resolution $\tilde f_n: \tilde \sX_n\to \Modulin$ which is a
    maximal family of $n$-dimensional projective CY manifolds.
    \item [(ii)] $\varsigma(\tilde f_n)=1$. Consequently,
    Shafarevich's conjecture holds for $\tilde f_n$.
    \item [(iii)] The family $\tilde f_3$ admits a partial
    compactification to a Shimura family over an arithmetic quotient
    of $\B^3$.
    \item [(iv)] The families $\tilde f_5,\tilde f_9$ have a sub
    $\Q$-PVHS which uniformize a Zariski open subset of an
    arithmetic ball quotient.
\end{itemize}
\end{theorem}

{\bf Acknowledgement:} We would like to thank Guitang Lan for
helpful discussions, particularly in Lemma \ref{symmetric map
induces isomorphism on the moduli}. Our special thanks go to Igor
Dolgachev who has drawn our attention to the $n=3$ case of the
paper.

\section{The cyclic cover and its crepant resolution}\label{section:crepant resolution}
The meaning of letters in the tuple $(n,m,r)$ will be fixed
throughout the paper: $n$ is a natural odd number $\geq 3$, $m=n+3$
and $r=\frac{m}{2}$.
\begin{remark}
The technique of \emph{this} section can be applied equally to a
tuple $(n,m,r)$ where $n$ is a natural number, $r$ a positive factor
of $m$ and $m=n+1+\frac{m}{r}$, and yields the same result as the
special case.
\end{remark}
\subsection{The cyclic cover of $\P^{n}$}\label{subsection:branch cover}
An hyperplane arrangement $\mathfrak A=(H_1,\cdots,H_m)$ in $\P^n$
is said to be in general position if no $n+1$ hyperplanes in
$\mathfrak A$ do meet. One constructs the $r$-fold cyclic cover
$\pi:X\rightarrow \P^{n}$ branched along $H=H_{1}+\cdots+ H_{m}$ as
follows: For the line bundle $L=\mathcal{O}_{\P^n}(2)$ over $\P^n$,
one denotes ${\rm Tot}(L)$ for the total space of $L$. There is the
tautological section $s\in \Gamma({\rm Tot}(L),p^{*}L)$ of the
pull-back of $L$ via the natural projection $p: {\rm Tot}(L)\to
\P^n$. Suppose the hyperplane $H_{i},1\leq i\leq m$ is defined as
the zero locus of the section $s_{i}\in
\Gamma(\P^{n},\mathcal{O}(1))$, then we have the pull-back section
$p^{*}s_{i}\in \Gamma({\rm Tot}(L),p^{*}\mathcal{O}(1))$. The
$r$-fold cover $X$ is defined as the zero locus of the section
$$s^{r}-p^{*}s_{1}\otimes \cdots \otimes p^{*}s_{m}\in \Gamma({\rm Tot}(L),p^{*}\mathcal{O}(m)).$$
The natural projection $p$ induces the one $\pi: X\to \P^n$. It is
generically \'{e}tale and Galois, whose Galois group is cyclic of
order $r$, and singular over $H$. The variety $X$ is projective with
trivial $K_X$. It is a singular CY variety whose singular locus
under $\pi$ is exactly the singularity of $H$.

\subsection{The crepant resolution}
In this paragraph we aim to obtain a good smooth model of the cyclic
cover $X$. First recall the order function on a smooth variety. Let
$M$ be a smooth variety over an algebraically closed field and
$\mathcal{I}$ be an ideal sheaf on $M$. For any point $x\in M$, the
order of $\mathcal{I}$ at $x$ is defined as ${\rm
Ord}_{x}\mathcal{I}:=max\{r| \mathcal{I}_{x}\subset
m_{x}^{r}\sO_{M,x}\}$, where $m_{x}$ is the maximal ideal of the
local ring $\sO_{M,x}$. For a smooth and irreducible closed
subvariety $Z$ of $M$, the order of $\mathcal{I}$ along $Z$ is
defined as ${\rm Ord}_{Z}\mathcal{I}:= {\rm Ord}_{p}\mathcal{I}$,
where $p\in Z$ is the generic point of $Z$. We have the following
well-known formula for the canonical bundle under a blow-up.
\begin{lemma}\label{crepant-proposition}
Let $M$ be a smooth variety over an algebraically closed field and
$X\subset M$ be a (possibly singular and non-reduced) hypersurface.
Suppose $\mathcal{I}(X)$ is the defining ideal sheaf of $X$ in $M$.
Let $Z\subset X$ be a smooth and irreducible closed subvariety of
$M$, with codimension ${\rm codim}(Z,M)=n$. Suppose that the order
of $\mathcal{I}(X)$ along $Z$ is ${\rm Ord}_{Z}\mathcal{I}(X)=r$.
Consider the blow-up of $M$ along $Z$:
$\tilde{M}\xrightarrow{\tilde{\pi}}M$. Let
$\tilde{X}=Bl_{Z}X\xrightarrow{\pi}X$ be the strict transform of $X$
and $E=\tilde{\pi}^{-1}(Z)$ the exceptional divisor. Then we have
the following formula relating the canonical bundles of $\tilde{X}$
and $X$: $$K_{\tilde{X}}\simeq
\pi^{*}K_{X}+(n-r-1)\sO_{\tilde{M}}(E)|_{\tilde{X}}.$$ In
particular, if ${\rm codim}(Z,M)=2$ and ${\rm
Ord}_{Z}\mathcal{I}(X)=1$, then $K_{\tilde{X}}\simeq \pi^{*}K_{X}$.
\end{lemma}
\begin{proof}
The proof is a direct application of the adjunction formula.
\end{proof}
In order to fix notations, we recall the following definition.
\begin{definition}\label{transversal-defn}
If $M$ is a smooth variety, and $E_{1},\cdots, E_{k}$ are divisors
of $M$, we say $E_{1},\cdots, E_{k}$ meet transversally at a closed
point $x\in M$ if one can choose a regular system of parameters
$z_{1},\cdots, z_{l}\in \mathcal{O}_{M,x}$ at $x$ such that for each
$1\leq i\leq k$
\begin{itemize}
  \item [(1)] either $x\not\in E_{i}$, or
  \item [(2)] $E_{i}=(z_{c(i)}=0)$ in a neighborhood of $x$ for some $c(i)$, and
  \item [(3)] $c(i)\neq c(i^{'})$ if $i\neq i^{'}$.
\end{itemize}
and the regular system of parameters $z_{1},\cdots, z_{l}\in
\mathcal{O}_{M,x}$ is called a coordinate system at $x$ admissible
to $E_{1},\cdots, E_{k}$. $E_{1},\cdots, E_{k}$ are said to meet
transversally if they meet transversally at each closed point of
$M$. In this case, $D=\sum_{i=1}^{k}E_{i}$ is called a simple normal
crossing divisor on $M$ and a subvariety $Z\subset M$ is said to
meet transversally with $E_{1},\cdots, E_{k}$ if at each closed
point $x\in Z$, one can choose  $z_{1},\cdots, z_{l}$ as above such
that in addition
 \begin{itemize}
  \item [(4)] $Z=(z_{j_{1}}=\cdots=z_{j_{s}}=0)$ for some $j_{1},\cdots, j_{s}$, again in some open neighborhood of $x$.
  \end{itemize}
In particular, $Z$ is smooth, and some of the $E_{i}$ are allowed to
contain $Z$.
\end{definition}
Let $M$ be a smooth variety, and $\mathfrak{E}=\{E_{1},\cdots,
E_{s}\}$, $\mathfrak{F}=\{F_{1},\cdots, F_{t}\}$ be sets of smooth
divisors of $M$ such that the $s+t$ divisors $E_{1},\cdots, E_{s},
F_{1},\cdots, F_{t}$ meet transversally on $M$. A reduced and
irreducible hypersurface $X\subset M$ is called pre-binomial with
respect to $(\mathfrak{E}, \mathfrak{F})$ if for any closed point
$x\in X$,
\begin{itemize}
\item[(1)] either $X, E_{1},\cdots, E_{s}, F_{1},\cdots, F_{t}$ meet transversally at $x$, or
\item[(2)] there exists a coordinate system $(y_{1},\cdots, y_{m},x_{1},\cdots, x_{n})$ at $x$ admissible to
$E_{1},\cdots, E_{s}, F_{1},\cdots, F_{t}$ (refer to Definition
\ref{transversal-defn} for the notation) such that
 \begin{itemize}
 \item the defining equation of $X$ is
  $$y_{1}^{a_{1}}\cdots y_{p}^{a_{p}}-x_{1}\cdots x_{q}=0$$ in a nonempty open neighborhood of $x$, where the integers satisfy $$1\leq p\leq m, 1\leq q\leq n, a_{1}\geq 1,\cdots, a_{p}\geq 1,$$ and
  \item for each $1\leq i\leq p$,  $y_{i}=0$ is a defining equation for some $E_{c_{i}}\in \mathfrak{E}$ in a nonempty open neighborhood of $x$, and
  \item for each $1\leq j\leq q$, $x_{j}=0$ is a defining equation for some $F_{d_{j}}\in \mathfrak{F}$ in a nonempty open neighborhood of $x$.
 \end{itemize}
\end{itemize}
If $X$ is a pre-binomial hypersurface of $M$ with respect to
$(\mathfrak{E}, \mathfrak{F})$, keeping the notations in the above
definition, then for any closed point $x\in X$, for any $E\in
\mathfrak{E}$, $F\in \mathfrak{F}$, define
\begin{equation}\notag
\begin{split}
e(E,x)&=\left\{
         \begin{array}{ll}
           a_{i}, & \hbox{if $E=E_{c_{i}}$ for some $i$ in the above definition ;} \\
           0, & \hbox{otherwise.}
         \end{array}
       \right.\\
e(F,x)&=\left\{
         \begin{array}{ll}
           1, & \hbox{if $F=F_{d_{j}}$ for some $j$ in the above definition ;} \\
           0, & \hbox{otherwise.}
         \end{array}
       \right.
\end{split}
\end{equation}
It is not difficult to verify that the above definitions of $e(E,x),
e(F,x)$ do not depend on the choice of local
coordinates, so they are well-defined nonnegative integers.\\
Let $M$, $\mathfrak{E}$, $\mathfrak{F}$ be as above, then a
hypersurface $X$ of $M$ is called binomial with respect to
$(\mathfrak{E},\mathfrak{F})$ if it is pre-binomial with respect to
$(\mathfrak{E},\mathfrak{F})$ and satisfies the following two
additional conditions:
\begin{itemize}
\item[(3)] $\forall E\in \mathfrak{E}\cup \mathfrak{F}$, $\forall  \textmd{ closed points }x_{1}, x_{2}\in E\cap X$, $e(E,x_{1})=e(E,x_{2})$.
\item[(4)] $\forall E\in \mathfrak{E}$, $\forall F\in \mathfrak{F}$, if there exists a closed point  $x\in X$ such that $e(E,x)>0$ and $e(F,x)>0$, then $E\cap F\subset X$.
\end{itemize}
Our key observation is the following stable proposition for binomial
hypersurfaces.
\begin{proposition}\label{stable prop of binomial hypersurfaces}
Let $M$, $\mathfrak{E}=\{E_{1},\cdots, E_{s}\}$,
$\mathfrak{F}=\{F_{1},\cdots, F_{t}\}$ be as above. Suppose $X$ is a
hypersurface of $M$ binomial with respect to $(\mathfrak{E},
\mathfrak{F})$ and there exists a closed point $x\in X$ such that
$e(E_{1},x)>0$ and $e(F_{1},x)>0$. Let
$M_{1}=Bl_{Z}M\xrightarrow{\pi}M$ be the blow-up of $M$ along
$Z=E_{1}\cap F_{1}\subset X$ and $X_{1}, E_{1}^{'},\cdots,E_{s}^{'},
F_{1}^{'},\cdots, F_{t}^{'}$ be the strict transforms of $X,
E_{1},\cdots,E_{s}, F_{1},\cdots, F_{t}$ respectively. Let
$\mathfrak{E}_{1}=\{ E_{1}^{'},\cdots, E_{s}^{'},\newline
E_{s+1}^{'}\}$, $\mathfrak{F}_{1}=\{F_{1}^{'},\cdots, F_{t}^{'}\}$,
where $E_{s+1}^{'}=\pi^{-1}(Z)$ is the exceptional divisor. Then
\begin{itemize}
\item[(1)]$\mathfrak{E}_{1}\cup \mathfrak{F}_{1}$ is a set of smooth divisors meeting transversally on $M_{1}$, and $X_{1}$ is a hypersurface of $M_{1}$ binomial with respect to $(\mathfrak{E}_{1}, \mathfrak{F}_{1})$.
\item[(2)]Each irreducible component of the singular locus of $X$, say $Sing(X)$, has the form $E_{i_{1}}\cap F_{j_{1}}\cap F_{j_{2}}\subset X$ or $E_{i_{1}}\cap E_{i_{2}}\cap F_{j_{1}}\cap F_{j_{2}}\subset X$, for $1\leq i_{1}\neq i_{2}\leq s$, $1\leq j_{1}\neq j_{2}\leq t$.
\item[(3)]The induced morphism $X_{1}\xrightarrow{\pi_{1}}X$ satisfies  that $\pi_{1}^{-1}(E_{1}\cap F_{1}\cap Sing(X))\rightarrow E_{1}\cap F_{1}\cap Sing(X)$ is a $\P^{1}-$bundle, and $\pi_{1}^{-1}(X-E_{1}\cap F_{1}\cap Sing(X))\rightarrow X-E_{1}\cap F_{1}\cap Sing(X)$ is an isomorphism.
\item[(4)]Let $\mathfrak{T}=\{E_{i_{1}}^{'}, \cdots, E_{i_{k}}^{'}, F_{j_{1}}^{'},\cdots, F_{j_{l}}^{'}\}\subset \mathfrak{E}_{1} \cup \mathfrak{F}_{1}$ be a subset of $\mathfrak{E}_{1} \cup \mathfrak{F}_{1}$ satisfying $V=E_{i_{1}}^{'}\cap \cdots\cap E_{i_{k}}^{'}\cap F_{j_{1}}^{'}\cap\cdots\cap F_{j_{l}}^{'}\subset X_{1}$, then $V=\emptyset $ if   $\{E_{1}^{'}, F_{1}^{'}\}\subset \mathfrak{T}$. Suppose $V\neq \emptyset$, then
    \begin{equation}\notag
    \pi(V)=\left\{
             \begin{array}{ll}
               E_{i_{1}}\cap\cdots\cap E_{i_{k}}\cap F_{j_{1}}\cap \cdots \cap F_{j_{l}}, & \hbox{  if $E_{s+1}^{'}\notin \mathfrak{T}$
;} \\
               E_{1}\cap F_{1}\cap E_{i_{2}}\cap \cdots\cap E_{i_{k}}\cap F_{j_{1}}\cap \cdots \cap F_{j_{l}}, & \hbox{ if  $E_{i_{1}}^{'}=E_{s+1}^{'}$.}
             \end{array}
           \right.
   \end{equation}
\item[(5)] Notations as in $(4)$. If $\{E_{1}^{'}, F_{1}^{'}\}\cap \mathfrak{T}\neq \emptyset$, then the induced morphism $V\xrightarrow{\tilde{\pi}} \pi(V)$ is an isomorphism. If $\{E_{1}^{'}, F_{1}^{'}\}\cap \mathfrak{T}= \emptyset$, then the induced morphism $V\xrightarrow{\tilde{\pi}} \pi(V)$ satisfies that $\tilde{\pi}^{-1}(E_{1}\cap F_{1}\cap \pi(V))\rightarrow E_{1}\cap F_{1}\cap \pi(V)$ is a $\P^{1}-$bundle, and $\tilde{\pi}^{-1}(\pi(V)-E_{1}\cap F_{1})\rightarrow \pi(V)-E_{1}\cap F_{1}$ is an isomorphism.
\end{itemize}
\end{proposition}
\begin{proof}
The verification is straightforward in local coordinates.
\end{proof}
Let $M$, $\mathfrak{E}=\{E_{1},\cdots, E_{s}\}$,
$\mathfrak{F}=\{F_{1},\cdots, F_{t}\}$ be as above. Suppose $X$ is a
hypersurface of $M$ binomial with respect to $(\mathfrak{E},
\mathfrak{F})$. In order to get a resolution of $X$, we define a
function on $X$ to measure the singularities on it. For any closed
point $x\in X$, define $g_{1}(x)=\sum_{F\in \mathfrak{F}}e(F,x)$,
$g_{2}(x)=\sum_{E\in \mathfrak{E}}e(E,x)$. Then we get a function on
$X$:
\begin{equation}\notag
g: X\rightarrow \mathbb{N}\times\mathbb{N},\quad  x\mapsto
(g_{1}(x),g_{2}(x)).
\end{equation}
It is easy to see that if $g(x)\neq (0,0)$, then $g^{-1}(g(x))$ is a
disjoint union of smooth  subvarieties of $X$. In this case, let
$n(x)$ be the number of irreducible components of $g^{-1}(g(x))$.
Then we define the following function:
\begin{equation}\notag
\begin{split}
f: X&\rightarrow \mathbb{N}\times\mathbb{N}\times\mathbb{N}\\
f(x)&=\left\{
       \begin{array}{ll}
         (g_{1}(x),g_{2}(x),n(x)), & \hbox{if $g(x)\neq (0,0)$;} \\
         0, & \hbox{if $g(x)=(0,0)$.}
       \end{array}
     \right.
\end{split}
\end{equation}
Given the lexicographic order,
$\mathbb{N}\times\mathbb{N}\times\mathbb{N}$ is a well-ordered set,
i.e. for $$(a_{1},b_{1},c_{1}),(a_{2},b_{2},c_{2})
\in\mathbb{N}\times\mathbb{N}\times\mathbb{N},$$
$(a_{1},b_{1},c_{1})<(a_{2},b_{2},c_{2})$ if and only if
$a_{1}<a_{2}$, or $a_{1}=a_{2}$ and $b_{1}<b_{2}$, or $a_{1}=a_{2}$
and $b_{1}=b_{2}$ and $c_{1}<c_{2}$. Then it is easy to see that $f$
is an upper semi-continuous function on $X$. Therefore we get the
following algorithm $(*)$ to resolve the singularity of $X$:
\begin{itemize}
\item[(0)] If the maximal value $max_{x\in X}f(x)=(0,0,0)$, then $X$ is already a smooth variety meeting transversally with the divisors in $\mathfrak{E}\cup \mathfrak{F}$.
\item[(1)] If $max_{x\in X}f(x)>(0,0,0)$, take any closed point $x\in X$ such that $f(x)$ attains the maximal value of $f$. It is not difficult to see that we can choose $E_{i}\in \mathfrak{E}$, $F_{j}\in \mathfrak{F}$ such that $e(E_{i},x)>0$ and  $e(F_{j},x)>0$. Then blow up $M$ along $Z= E_{i}\cap F_{j}$ (we have $Z= E_{i}\cap F_{j}\subset X$ by the definition of binomial hypersurfaces). Let $E$ be the exceptional divisor. Let $M_{1}=Bl_{Z}M$, and $X_{1},E_{1}^{'},\cdots,E_{s}^{'}, F_{1}^{'},\cdots, F_{t}^{'}$ be the strict transforms of $X, E_{1},\cdots,E_{s}, F_{1},\cdots, F_{t}$ respectively. Let $\mathfrak{E}_{1}=\{E, E_{1}^{'},\cdots, E_{s}^{'}\}$, $\mathfrak{F}_{1}=\{F_{1}^{'},\cdots, F_{t}^{'}\}$, then according to Proposition \ref{stable prop of binomial hypersurfaces}, $X_{1}$ is a hypersurface of $M_{1}$ binomial with respect to $(\mathfrak{E}_{1},\mathfrak{F}_{1})$. So we can define a function $f_{1}:X_{1}\rightarrow \mathbb{N}\times\mathbb{N}\times\mathbb{N}$ in the same way as above. Let $\pi_{1}: X_{1}\rightarrow X$ be the blow-up morphism. Then it is direct to verify:
\begin{itemize}
\item for any  point $x\in X_{1}$, $f_{1}(x)\leq f(\pi_{1}(x))$, and
\item the maximal value drops strictly, i.e. $max_{x\in X_{1}}f_{1}(x)< max_{x\in X}f(x)$.
\end{itemize}
Note also that since $Z$ has codimension $2$ everywhere in $M$, and
$X$ has order $1$ at the  generic point of each irreducible
component of  $Z$, we have $K_{X_{1}}\simeq \pi_{1}^{*}K_{X}$,
according to Lemma \ref{crepant-proposition}.
\item[(2)] If $max_{x\in X_{1}}f_{1}(x)>(0,0,0)$, then continue to blow up $M_{1}$ and get $$M_{2}, X_{2}, \mathfrak{E}_{2},\mathfrak{F}_{2},f_{2}.$$
$\cdots$
\end{itemize}
We summarize the above discussions in the following theorem:
\begin{theorem}\label{resolution theorem}
Let $M$, $\mathfrak{E}=\{E_{1},\cdots, E_{s}\}$,
$\mathfrak{F}=\{F_{1},\cdots, F_{t}\}$ be as above. Suppose $X$ is a
hypersurface of $M$ binomial with respect to $(\mathfrak{E},
\mathfrak{F})$. Then the above algorithm $(*) $ terminates after
finite steps. Suppose it terminates after $N$ steps, then we get
$M_{N}, X_{N}, \mathfrak{E}_{N},\mathfrak{F}_{N}$ such that $X_{N}$
is a smooth hypersurface of $M_{N}$ meeting transversally with the
set of divisors  $\mathfrak{E}_{N}\cup \mathfrak{F}_{N}$. Moreover,
let $\pi=\pi_{N}\circ\cdots\circ\pi_{1}:X_{N}\rightarrow X$ be the
composition of blow-up morphisms, then
\begin{itemize}
\item $\pi$ is crepant, i.e. $K_{X_{N}}\simeq \pi^{*}K_{X}$;
\item  $\pi $ is a strong resolution of $X$, i.e. if $U=X-Sing(X)$ is the regular part of $X$, then $\pi$ induces an isomorphism $\pi^{-1}(U)\xrightarrow{\sim} U$;
\item $\pi$ is a projective morphism, moreover, it is a composition of blow-ups along smooth centers.
\end{itemize}
\end{theorem}
\begin{proof}
Most of the theorem follows from the above discussions. We just
explain why $\pi$ is a strong resolution. Note that in each blow-up
step the smooth center $Z_{i}$ is contained completely in $X_{i}$.
So in the regular part of $X_{i}$, we just blow up a Cartier
divisor. Therefore the regular part of $X_{i}$ remains unchanged.
\end{proof}
Now we give an application of Theorem \ref{resolution theorem}.
Suppose $Q$ is a smooth variety, and $D=\sum_{j=1}^{t}F_{j}$ is a
simple normal crossing divisor on $Q$ defined by the section
$s_{D}\in \Gamma(Q,\mathcal{O}_{Q}(D))$. Let $a$ be a positive
integer and $L\in {\rm Pic}(Q)$ such that $aL= D$. Then in the total
space of $L$: $M={\rm Tot}(L)\xrightarrow{p} Q$, we have the
tautological section $s_{0}\in \Gamma(M, p^{*}L)$.  The hypersurface
$X$ of $M$ defined by the equation $s_{0}^{a}=p^{*}s_{D}$ is called
the $a$-fold cyclic cover of $Q$ branched along $D$. It is easy to
verify that $X$ is a hypersurface of $M$ binomial with respect to
$(\mathfrak{E},
\mathfrak{F})=(\{E_{0}=(s_{0}=0)\},\{p^{*}F_{1},\cdots,
p^{*}F_{t}\})$. So we can apply Theorem \ref{resolution theorem} to
get a crepant resolution of $X$. That is, we have the following:
\begin{corollary}\label{branched covering resolution}
Suppose $Q$ is a smooth variety and $D=\sum_{j=1}^{t}F_{j}$ is a
simple normal crossing divisor on $Q$, for any $a\geq 1$, if there
exists $L\in Pic(Q)$ such that $aL= D$, then the $a$-fold cyclic
cover of $Q$ branched along $D$ admits a crepant resolution, which
can be obtained by applying the crepant resolution algorithm $(*)$.
\end{corollary}
For the cyclic cover $X$ constructed in \S\ref{subsection:branch
cover}, we can simply apply the above result to obtain a crepant
resolution $\sigma: \tilde X \to X$.
\subsection{The middle cohomology does not change under resolution}
In this paragraph we show that $\sigma: \tilde X\to X$ induces an
isomorphism on the middle cohomologies. To start with, we show first
that the Hodge structure $H^n(X,\Q)$ is actually pure. This is
mostly easy if one notices that there is another CY $Y$ arising from
the same hyperplane arrangement as $X$ which is indeed
\emph{smooth}. We follow \S2.2 \cite{GSSZ} for the construction of
$Y$, the Kummer cover associated with $\mathfrak A$. Let $A$ be a
matrix whose columns define the hyperplane arrangement $\mathfrak A$
and $B=(b_{ij})$ a matrix fitting into a short exact sequence
$$
0 \to
\C^{m-2}\stackrel{A}{\longrightarrow}\C^{m}\stackrel{B}{\longrightarrow}\C^2\to
0.
$$
We define $Y$ to be the complete intersection of $\frac{m}{r}=2$
degree $r$ hypersurfaces in $\P^{m-1}=\Proj\ \C[z_0:\cdots:z_{m-1}]$
defined by
$$
b_{i1}z_{0}^{r}+\cdots+b_{im}z_{m-1}^r=0,\quad i=1,2.
$$
Using the Jacobian criterion, one sees easily that $Y$ is smooth,
and by the adjunction formula has trivial canonical bundle. The
structure of $Y$ as a Kummer cover is seen as follows: Put
$G_{1}=\oplus_{i=1}^{m}\Z/r/\triangle(\Z/r)$. Here $\triangle$
denotes for the diagonal embedding $\Z/r\to \oplus_{i=1}^{m}\Z/r$.
For $a=(a_{0},\cdots, a_{m})\in G_{1}$, we define an automorphism
$\sigma_{a}: \P^{m-1}\rightarrow \P^{m-1}$ by:
\begin{equation}\notag
\sigma_{a}(z_{0}:\cdots:z_{m-1})=(\zeta_r^{a_{0}}z_{0}:\cdots:\zeta_r^{a_{m-1}}z_{m-1})
\end{equation}
where $\zeta_r$ is a primitive $r$-th root of unit. The matrix $A$
defines a linear embedding $j: \P^n\to \P^{m-1}$ of projective
spaces. Similar to Lemma 2.4 loc. cit., one sees that the map
$$
\P^{m-1}\to \P^{m-1}, (z_0:\cdots:z_{m-1})\mapsto (z_0^{r}:
\cdots:z_{m-1}^r)
$$
realizing $Y$ as the Kummer cover $Y\to j(\P^n)$ with Galois group
$G_1$ and with branch locus $j(H)$. The group $G_{1}$ contains a
distinguished normal subgroup $N_{1}\lhd G_{1}$ of index $r$, the
kernel of the map $a\mapsto \sum_{i=0}^{m-1}a_{i}$. We state the
following result whose proof is referred to Proposition 2.5 loc.
cit..
\begin{lemma}\label{HS of X as a quotient of Y}
It holds that $X\simeq Y/N_1$. Consequently, there is an isomorphism
of pure polarized $\Q$-Hodge structures
$$H^n(X,\Q)\simeq H^n(Y,\Q)^{N_1}.$$
\end{lemma}
\begin{proposition}\label{thm:Hodge structures under crepant resolution}
If $p,q\geq 0$ and $p\neq q$, then $\sigma$ induces an isomorphism
$$\sigma^{*}: H^{p,q}(X)\xrightarrow{\sim}H^{p,q}(\tilde{X}).$$ In
particular, one has the isomorphism of $\Q$-PVHS:
$$
\sigma^{*}: H^{n}(X, \Q)\xrightarrow{\sim}H^{n}(\tilde{X}, \Q).
$$
\end{proposition}
We need some lemmas.
\begin{lemma}\label{lemma:Steenbrink and Peters book lemma}
Let $f: \tilde{X}\rightarrow X$ be a proper modification with
discriminant $D$. Put $E=f^{-1}(D)$. For $p, q\geq 0$, if
$H^{p,q}(X)=H^{p,q}(E)=0$, then $H^{p,q}(\tilde{X})=0$.
\end{lemma}
\begin{proof}
This follows directly from Corollary-Definition $5.37$ in \cite{PS}.
\end{proof}
\begin{lemma}\label{lemma:h^{p,q}=0lemma}
Let $\pi: \tilde{V}\rightarrow V$ be a surjective morphism between
projective varieties, $Z\subset V$ a closed subvariety such that
$\pi^{-1}(Z)\xrightarrow{\pi}Z $ is a $\P^{1}-$bundle and
$\pi^{-1}(V-Z)\xrightarrow{\pi}V-Z$ is an isomorphism. For $p, q\geq
0$ and $p\neq q$, if $H^{p,q}(Z)=0$, then the natural homomorphism
$H^{p,q}(V)\xrightarrow{\pi^{*}} H^{p,q}(\tilde{V})$ is surjective.
\end{lemma}
\begin{proof}
This follows from Lemma \ref{lemma:Steenbrink and Peters book lemma}
and the Leray-Hirsch Theorem for the $\P^{1}$-bundle
$\pi^{-1}(Z)\xrightarrow{\pi}Z $.
\end{proof}
We come to the proof of Proposition \ref{thm:Hodge structures under
crepant resolution}.
\begin{proof}
Set $M={\rm Tot}(L)$ and $E_{1}=\{s=0\}, F_{j}=\{p^{*}s_{j}=0\}$
divisors in $M$. Put
$$\mathfrak{E}=\{E_{1}\}, \mathfrak{F}=\{F_{1},\cdots, F_{m}\}.$$
Then $X$ is a binomial hypersurface of $M$ with respect to
$(\mathfrak{E},\mathfrak{F})$. The crepant resolution algorithm
applied to the data $X, M, (\mathfrak{E},\mathfrak{F})$ provides us
with the blow-up sequence
\begin{equation}\notag
\tilde{X}=X_{N}\xrightarrow{\sigma_{N}}X_{N-1}\xrightarrow{\sigma_{N-2}}\cdots
\xrightarrow{\sigma_{1}}X
\end{equation}
such that for any $i=1,2,\cdots, N$, $X_{i}$ is a binomial
hypersurface of $M_{i}$ with respect to
$(\mathfrak{E}_{i},\mathfrak{F}_{i})$, and $\tilde{X}=X_{N}$ is
smooth and meet transversally to the divisors in
$\mathfrak{E}_{N}\cup \mathfrak{F}_{N}$. An induction using
Proposition \ref{stable prop of binomial hypersurfaces} and Lemma
\ref{lemma:h^{p,q}=0lemma} shows that for any $i=1,2,\cdots, N$, for
any elements $D_{1},\cdots, D_{k}\in
\mathfrak{E}_{i}\cup\mathfrak{F}_{i}$, if $D_{1}\cap\cdots \cap
D_{k}\subset X_{i}$, then $H^{p,q}(D_{1}\cap\cdots \cap D_{k})=0$,
for any $p, q\geq 0$ and $p\neq q$. Then a further induction using
$(3)$ of Proposition \ref{stable prop of binomial hypersurfaces} and
Lemma \ref{lemma:h^{p,q}=0lemma} shows that
$H^{p,q}(X)\xrightarrow{\sigma^{*}}H^{p,q}(\tilde{X})$ is
surjective, for any $p, q\geq 0$ and $p\neq q$, where
$\sigma=\sigma_{N}\circ\cdots\circ\sigma_{1}$. Theorem 5.41 in
\cite{PS} shows that
$H^{p,q}(X)\xrightarrow{\sigma^{*}}H^{p,q}(\tilde{X})$ is injective,
for any $p,q\geq 0$. So we finally get that
$H^{p,q}(X)\xrightarrow{\sigma^{*}}H^{p,q}(\tilde{X})$ is an
isomorphism, for any $p,q\geq 0$ and $p\neq q$.
\end{proof}
\begin{corollary}
Let $\pi: X\rightarrow \P^{n}$ be the $r$-fold cyclic cover of
$\P^{n}$ branched along $m$ hyperplanes in general position in
\S\ref{subsection:branch cover}. Let $\tilde{X}\xrightarrow{\sigma}
X$ be the crepant resolution constructed after Corollary
\ref{branched covering resolution}. The obtained $\tilde{X}$ is a
smooth projective CY manifold.
\end{corollary}
\begin{proof}
By Theorem \ref{resolution theorem}, $\tilde{X}$ is smooth. Also, as
the morphism $\sigma: \tilde{X}\rightarrow X$ is projective by the
same result, $\tilde{X}$ is projective whose canonical bundle is
trivial. By Proposition \ref{thm:Hodge structures under crepant
resolution}, for any $0<i<n$, we have $$ H^{i}(\tilde{X},
\sO_{\tilde{X}})\simeq H^{0,i}(\tilde{X})\simeq
H^{0,i}(X)=H^{0,i}(Y)^{N_1}=0.$$ This shows the result.
\end{proof}

\section{The Hodge structure of the cyclic cover}
Let $\pi: X\to \P^n$ be the $r$-fold cyclic cover branched along
$H=\sum_{i=1}^{m}H_i$ where $\{H_1,\cdots,H_{m}\}$ is a hyperplane
arrangement of $\P^n$ in general position. In this section we
investigate the Hodge structure $H^n(X,\Q)$. We first record its
Hodge numbers.
\begin{lemma}\label{Hodge number}
Let $h^{p,q}=\dim H^{p,q}(X)$ for $p,q\geq 0$ and $p+q=n$. One has
\begin{equation}\notag
h^{p,q}(X)=\left\{
             \begin{array}{ll}
               q+1, & \hbox{ if $q$ is even;} \\
               n+1-q, & \hbox{ if $q$ is odd.}
             \end{array}
           \right.
\end{equation}
\end{lemma}
\begin{proof}
By Lemma \ref{HS of X as a quotient of Y}, we can derive the Hodge
numbers of $X$ from those of $Y$, which is a smooth complete
intersection. By the work of T. Terasoma \cite{T}, one can represent
the cohomology classes of $Y$ by a certain Jacobian ring, together
with an explicit description of the action of $G_1$. After the
computation has been implemented, we found that Lemma 8.2 \cite{DK}
actually contains our result (set
$\mu=(\frac{1}{r},\cdots,\frac{1}{r})$ in the cited lemma). The
detail is therefore omitted.
\end{proof}
\subsection{The cyclic cover of $\P^1$ branched along $m$ distinct
points}\label{cyclic cover of P^1} For $m$ distinct points $p_i,
1\leq i\leq m$ in $\P^1$, we consider the $r$-fold cyclic cover of
$\P^1$ branched along $\sum_ip_i$. Call it the curve $C$, and let
$p: C\to \P^1$ be the natural projection. Fix a generator $\iota\in
\Aut(C|\P^1)$. Let $\Q(\zeta_r)\subset \C$ be the $r$-th cyclotomic
field. The induced action $\iota^*$ on $V_{\Q}:=H^1(C,\Q)$ induces a
decomposition of $V_{\Q(\zeta_r)}:=V_{\Q}\otimes \Q(\zeta_r)$ into
direct sum of $\Q(\zeta_r)$-PHSs:
$$
V_{\Q(\zeta_r)}=\oplus_{i=0}^{r-1}V_i,
$$
where $V_i$ is the eigenspace of $\iota^*$ with eigenvalue
$\zeta_r^i$. The following lemma is well known.
\begin{lemma}\label{decomposition of weight one HS}
Let $V_i\otimes \C=V_i^{1,0}\oplus V_i^{0,1}$ be the induced Hodge
decomposition from the weight one PHS of $C$. It holds that
$\overline{V_i^{1,0}}=V_{r-i}^{0,1}$ and
$$
\dim V_0=0, \quad \dim V_i^{1,0}=2i-1, 1\leq i\leq r-1.
$$
\end{lemma}
\begin{proof}
See for example Lemma 4.2 \cite{Lo}.
\end{proof}
The next lemma is purposed for a later use.
\begin{lemma}\label{decomposition of n-th wedge product}
The weight $n$ $\Q$-PHS $\bigwedge^nV_{\Q}$ admits a decomposition
$$
\bigwedge^nV_{\Q}=W_{1,\Q}\bigoplus W_{2,\Q}
$$
with $W_{1,\Q}\otimes
\Q(\zeta_r)=\bigoplus_{i=1}^{r-1}\bigwedge^nV_i$. Furthermore, only
for $n=3,5,9$, one has a further decomposition of $\Q$-PHS
$$
W_{1,\Q}=W_{unif, \Q}\oplus W_{1,\Q}'
$$
such that
$$
W_{unif,\Q}\otimes \Q(\zeta_r)=\bigwedge^n V_1\bigoplus \bigwedge^n
V_{r-1}.
$$
\end{lemma}
\begin{proof}
Consider the Galois action on the decomposition
\begin{eqnarray*}
  \bigwedge^n(V_{\Q(\zeta_r)})&=& \bigwedge^n(\bigoplus_iV_i) \\
  &=& \bigoplus_i\bigwedge^nV_i\oplus \bigoplus_{I, |I|\geq 2}\bigwedge^{r_I}V_{I}.
\end{eqnarray*}
Here $I=(i_1,\cdots,i_{|I|})$ for $1\leq i_1<\cdots<i_{|I|}\leq n$
is a multi-index, $r_I=(r_1,\cdots,r_{|I|})$ is a sequence of
nonnegative integers satisfying $\sum r_i=n$ and
$$
\bigwedge^{r_I}V_I=\bigwedge^{r_1}V_{i_1}\otimes \cdots
\bigwedge^{r_{|I|}}V_{i_{|I|}}.
$$
It is clear that $\bigoplus_i\bigwedge^nV_i$ is invariant under the
Galois action. So is the other factor. Thus they underlie
$\Q$-subspaces of $\bigwedge^n V_{\Q}$. Consider furthermore the
Galois orbit of $\bigwedge^nV_1$. It contains at least
$\bigwedge^nV_{r-1}$, and $\bigwedge^nV_1\oplus \bigwedge^nV_{r-1}$
has a sub $\Q$-structure iff the Galois orbit contains no more
direct summand. As the Galois group is isomorphic to the unit group
of $\Z/r$ and it is an elementary fact that the unit group is of
order two only if $r=3,4,6$, it follows that only when $n=3,5,9$,
$\bigwedge^nV_1\oplus \bigwedge^nV_{r-1}$ underlie a sub
$\Q$-structure. The lemma is proved.
\end{proof}
Now we perform a similar construction to the one taken in \S2.3
\cite{GSSZ}. Let $\gamma: (\P^1)^n\to \Sym^n(\P^1)=\P^n$ be the
Galois cover with Galois group $S_n$, the permutation group of $n$
letters, and the identification attaches to a divisor of degree $n$
the ray of its equation in $H^0(\P^1,\sO(n))$.
\begin{lemma}
Put $H_i=\gamma(\{p_i\}\times (\P^1)^{n-1})$. Then
$(H_1,\cdots,H_{m})$ is a hyperplane arrangement in $\P^n$ in
general position.
\end{lemma}
\begin{proof}
The divisors of degree $n$ in $\P^1$ containing a given point form a
hyperplane and, as a divisor of degree $n$ can not contain $n+1$
distinct points, no $n+1$ hyperplanes in the arrangement do meet.
\end{proof}
In this case, we have more: For any natural number $n$ (not
necessarily odd), we will show that any (ordered) $m$ hyperplane
arrangement in $\P^n$ is projectively equivalent to a(n) (ordered)
one arising from the above way. In fact, we will prove a stronger
statement. Let $\Modulione$ be the moduli space of ordered $n+3$
distinct points in $\P^1$ and similarly $\Modulin$ ordered $n+3$
hyperplane arrangements in $\P^n$ in general position.
\begin{lemma}\label{symmetric map induces isomorphism on the moduli}
The map $\gamma: (\P^1)^n\to \P^n$ induces an isomorphism
$$
\Gamma: \Modulione \simeq \Modulin, \quad [(p_1,\cdots,p_m)]\mapsto
[(H_1,\cdots,H_{m})].
$$
\end{lemma}
\begin{proof}
We adopt an elementary but direct proof. The two moduli spaces admit
affine descriptions: Any ordered $m=n+3$ distinct points in $\P^1$
is transformed into a unique tuple
$(0,t_1,\cdots,t_{n-1},\infty,1,t_n)$ with $t_i\in \C$ and
$\Modulione$ is therefore identified with the complement of the
following hyperplanes in $\C^n=\Spec\ \C[t_1,\cdots,t_n]$ defined by
$$
\{t_i, t_i-1, t_i-t_j, 1\leq i,j\leq n\}.
$$
Similarly, the following matrix represents a unique point in
$\Modulin$:
$$
\left(
  \begin{array}{ccccc}
    1 & \cdots & 0 & 1 & 1 \\
    0 & \ddots & 0 & 1 & s_1 \\
    \vdots & \ddots& \vdots & \vdots & \vdots \\
    0 & \cdots & 1 & 1 & s_n \\
  \end{array}
\right)_{n+1\times n+3},
$$
where $(s_1,\cdots,s_n)$ is a point in $\C^n=\Spec\
\C[s_1,\cdots,s_n]$ away from the union of hyperplanes
$$
\{s_i, s_i-1, s_i-s_j, 1\leq i,j\leq n\}.
$$
The following claim implies the lemma:
\begin{claim}
Under the above affine coordinates, one has
$$
\Gamma(t_1,\cdots,t_{n-1},t_n)=(\frac{t_{n}(t_1-1)}{t_1-t_{n}},\cdots,\frac{t_{n}(t_{n-1}-1)}{t_{n-1}-t_{n}},t_n).
$$
\end{claim}
\begin{proof}
The following Vandermonde-type matrix of size $(n+1)\times (n+3)$
give the defining equations of the hyperplane arrangement
corresponding to $$(0,t_1,\cdots,t_{n-1},\infty,1,t_n)$$ under
$\gamma$:
$$
A=\left(
  \begin{array}{ccccccc}
    1 & 1&\cdots &1 & 0 & 1 & 1 \\
    0 & t_1&\cdots&t_{n-1} & 0 & 1 & t_n \\
    \vdots&\vdots & \vdots&\vdots& \vdots & \vdots & \vdots \\
    0 &t_1^{n}& \cdots & t_{n-1}^{n}&1 & 1 & t_n^n \\
  \end{array}
\right).
$$
The first $n+1$ columns make the square matrix
$$
B=\left(
  \begin{array}{ccccc}
    1 & 1&\cdots &1 & 0  \\
    0 & t_1&\cdots&t_{n-1} & 0  \\
    \vdots&\vdots & \vdots&\vdots& \vdots  \\
    0 &t_1^{n}& \cdots & t_{n-1}^{n}&1  \\
  \end{array}
\right).
$$
By the Cramer's rule, one uses the determinant of a Vandermonde
matrix to determine the vector
$$
(\lambda_1,\cdots,\lambda_{n+1})^t=B^{-1}(1,\cdots,1)^t,
$$
as well as the vector
$$
(\mu_1,\cdots,\mu_{n+1})^t=B^{-1}(1,t_n,\cdots,t_{n}^n)^t.
$$
Put $D=\rm{diag}\{\lambda_1,\cdots,\lambda_{n+1}\}$. Then the
invertible $(n+1)\times (n+1)$ matrix
$P=\lambda_1\mu_1^{-1}D^{-1}B^{-1}$ transforms the hyperplane
arrangement defined by $A$ to the one by columns of the matrix
$$
\left(
  \begin{array}{ccccc}
    1 & \cdots & 0 & 1 & 1 \\
    0 & \ddots & 0 & 1 & \frac{t_{n}(t_1-1)}{t_1-t_{n}} \\
    \vdots & \cdots& \vdots & \vdots & \vdots \\
    0& \cdots& 0&1& \frac{t_{n}(t_{n-1}-1)}{t_{n-1}-t_{n}}\\
    0 & \cdots & 1 & 1 & t_n \\
  \end{array}
\right).
$$
\end{proof}
\end{proof}
Let $p: C\to \P^1$ as above. The $n$-fold product
$$
h: C^{n}\stackrel{p^n}{\longrightarrow}
(\P^1)^{n}\stackrel{\gamma}{\longrightarrow} \P^n
$$
is a Galois cover with Galois group $N \rtimes S_n$, where
$N=\langle\iota_1,\cdots,\iota_n\rangle$ is the group generated by
the cyclic automorphisms on factors. The group $N$ has a natural
normal subgroup $N'$ given by the kernel of the trace map
$$
N\simeq (\Z/r)^{\times n}\stackrel{\sum}{\longrightarrow}\Z/r.
$$
It has a set of generators $\{\delta_i\}_{1\leq i\leq r-1}$ with
$$ \delta_i=\langle id, \cdots,
\iota,\iota^{-1},\cdots, id\rangle,
$$
where $\iota$ appears at the $i$-th component and $\iota^{-1}$ the
$i+1$-th component and the identity elsewhere. Consider the quotient
$C^n/G$ with $G=N'\rtimes S_n$. Similar to Lemma 2.8 loc. cit., one
checks that the natural map $C^n/G\to \P^n$ induced by $h$ is a
Galois cover with Galois group $\Z/r$ and its branch locus is
exactly $H_1+\cdots+H_{m}$. As the Picard group of a projective
space has no torsion, one concludes that $C^n/G$ is isomorphic to
the $r$-fold cyclic cover $X_C$ of $\P^n$ branched along
$\sum_iH_i$.

\subsection{The Abel-Jacobi map and the Hodge structure of the cyclic cover}
Recall that (see e.g. Ch. 2 \cite{GH}) the Abel-Jacobi map $\phi:
C\to \Jac(C)=\C^g/\Lambda$ is defined by the period integral
$$
q\mapsto (\int_{q_0}^{q}\omega_1,\cdots,\int_{q_0}^{q}\omega_g),
$$
where $q_0\in C$ is a chosen base point and $\{\omega_i\}_{1\leq
i\leq g}$ is a basis of holomorphic one forms on $C$. Here $\Lambda$
denotes for the period lattice of $C$. Let $(z_1,\cdots,z_g)$ be the
standard coordinates of $\C^g$. Then one sees that $\phi^*$ induces
an isomorphism $H^1(\Jac(C),\Q)\simeq H^1(C,\Q)$ such that
$\phi^*(dz_i)=\omega_i, 1\leq i\leq g$. The Abel-Jacobi map induces
the natural morphism
$$
\phi_n: C^n\to \Jac(C),\  (q_1,\cdots,q_n)\mapsto
\sum_{i=1}^n\phi(q_i).
$$
We are in the situation to study the following diagram of morphisms:
$$
\xymatrix{
   &C^n \ar[d]_{\delta} \ar[r]^{\phi_n} & \Jac(C).  \\
   X_C\ar[r]^{\simeq} &  C^n/G  &    }
$$
Here the map $\delta: C^n\to C^n/G$ is the natural projection. Note
that after Lemma \ref{symmetric map induces isomorphism on the
moduli} any $r$-fold cyclic cover $X$ in \S\ref{subsection:branch
cover} is isomorphic to an $X_C$. The consequence of the morphisms
on the level of $\Q$-PHS is summarized in the following
\begin{proposition}\label{isomorphism of PHS with W}
The morphisms in the above diagram induces an isomorphism of
$\Q$-PHS:
$$
H^n(X_C,\Q)\simeq W_{1,\Q},
$$
where $W_{1,\Q}$ is the PHS in Lemma \ref{decomposition of n-th
wedge product}.
\end{proposition}
\begin{proof}
By labeling the curve factors of $C^n$ by $C_i, 1\leq i\leq n$, we
obtain from the K\"{u}nneth decomposition a decomposition of
$\Q$-PHS:
$$
H^n(C^n,\Q)=\bigotimes_{i=1}^{n}H^1(C_i,\Q)\oplus \bigotimes_{1\leq
i\neq j\leq n}[ H^0(C_i,\Q)\otimes
H^2(C_j,\Q)\otimes\bigotimes_{1\leq k\leq n, k\neq i,j}
H^1(C_k,\Q)].
$$
Denote $\kappa: H^n(C^n,\Q)\to \bigotimes_{i=1}^{n}H^1(C_i,\Q)$ for
the projection onto the first factor. We claim that the composite
$$
W_{1,\Q}\subset \bigwedge^nH^1(C,\Q)\stackrel{\phi^*}{\simeq}
H^n(\Jac(C),\Q)\stackrel{\phi_n^*}{\longrightarrow}H^n(C^n,\Q)\stackrel{\kappa}{\longrightarrow}
\bigotimes_{i=1}^{n}H^1(C_i,\Q)
$$
is injective and the image is $G$-invariant. Assuming the claim, the
result follows: As
$$
H^n(X_C,\Q)\simeq H^n(C^n/G,\Q)=H^n(C^n,\Q)^G
$$ and $W_{1,\Q}$ have the same dimensions by Lemmas
\ref{Hodge number}, \ref{decomposition of weight one HS} (both are
equal to $2(r-1)^2$), the injective map $W_{1,\Q}\hookrightarrow
H^n(C^n,\Q)^G$ of $\Q$-PHSs is indeed an isomorphism. Notice that by
Lemma \ref{decomposition of weight one HS} we can take the basis
$\{\omega_i\}_{1\leq i\leq g}$ in the above description of $\phi$ as
eigenvectors with respect to the action of $\iota$ on $V_{\C}$. This
implies that the corresponding $dz_i$ to $\omega_i$ is an
eigenvector with the same eigenvalue with respect to the induced
natural action of $\iota$ on $\Jac(C)$. This will make the
verification of the invariance of the resulting class under
$G$-action straightforward. Let $\pi_i: C^n\to C, 1\leq i\leq n$ be
the $i$-th projection. Then the induced map $\phi_n^*$ on the level
of differential one forms is given by
$$
\phi_n^*(dz_i)=\sum_{l=1}^{n}\pi_l^*\omega_i, \quad \phi^*(d\bar
z_i)=\sum_{l=1}^{n}\pi_l^*\bar \omega_i, \ 1\leq i\leq g.
$$
Therefore the map $\phi_n^*: H^n(\Jac(C),\C)\to H^n(C^n,\C)$ on the
degree $n$ cohomology groups is given by sending $ [dz_{i_1}\wedge
\cdots\wedge dz_{i_p}\wedge d\bar z_{j_1}\wedge \cdots \wedge d\bar
z_{j_{n-p}}]$ to
$$ [\sum_l
\pi_l^*\omega_{i_1}\wedge \cdots \sum_l \pi_l^*\omega_{i_p}\wedge
\sum_l \pi_l^*\bar \omega_{j_1}\wedge \cdots \wedge\sum_l
\pi_l^*\bar \omega_{j_{n-p}}].
$$
The bracket means the cohomology class of the differential form. Now
for each direct factor $\wedge^nV_k\subset W_{1,\Q}\otimes
\Q(\zeta_r)$, we claim that the image of its element under
$\kappa\circ\phi_n^*$ is invariant under $G$-action. Note that the
map $\phi_n$ factors as
$$
C^n\stackrel{/S_n}{\longrightarrow}S^nC\to \Jac(C).
$$
Thus it suffices to show the invariant property under the subgroup
$N'$ of $G$. It is also equivalent to show this property for
elements in $\wedge^n(V_k\otimes \C)$, which for dimension reason,
is just equal to
$$
\wedge^{2k-2}V_k^{1,0}\otimes \wedge^{n-2k+2}V_k^{0,1}\oplus
\wedge^{2k-1}V_k^{1,0}\otimes \wedge^{n-2k+1}V_k^{0,1}.
$$
Let $\{dz_{i_1},\cdots,dz_{i_{2k-1}}\}$ (resp. $\{d\bar
z_{j_1},\cdots,d\bar z_{j_{n-2k+2}}\}$) be the basis of $V_k^{1,0}$
(resp. $V_k^{0,1}$). Consider the image under $\phi_n^*$ of a
typical element (omitting $dz_{i_{2k-1}}$ in the wedge product)
$$[\alpha]:=[dz_{i_1}\wedge\cdots\wedge dz_{i_{2k-2}}\wedge d\bar z_{j_1}\wedge
\cdots d\bar z_{j_{n-2k+2}}]\in \wedge^{2k-2}V_k^{1,0}\otimes
\wedge^{n-2k+2}V_k^{0,1}.$$ For clarity we set
$$
\alpha_1=\omega_{i_1},\cdots,\alpha_{2k-2}=\omega_{i_{2k-2}},\alpha_{2k-1}=\bar
\omega_{j_1},\cdots,\alpha_{n}=\bar \omega_{j_{n-2k+2}}.
$$
Then it follows that
$$
\kappa\circ\phi_n^*([\alpha])=\sum_{\nu\in
S_n}(-1)^{\rm{Sign}(\nu)}[\pi_1^*\alpha_{\nu(1)}\wedge \cdots\wedge
\pi_n^*\alpha_{\nu(n)}].
$$
(For $n=3\mod 4$, it holds that
$\kappa\circ\phi_n^*([\alpha])=\phi_n^*([\alpha])$.) Now one sees
its invariance under $N'$-action immediately: One tests simply the
action of any generator $\delta_i$ of $N'$ and it acts on
$\pi_1^*\alpha_{\nu(1)}\wedge \cdots\wedge \pi_n^*\alpha_{\nu(n)}$
by multiplying $\zeta_{r}^k\zeta_{r}^{-k}=1$. This completes the
proof.
\end{proof}

\section{Maximal families of CY manifolds with length one Yukawa coupling}
Our aim in this section is to exhibit families of CY manifolds with
claimed properties, and make some complements to these families at
the end.\\
Recall that in the proof of Lemma \ref{symmetric map induces
isomorphism on the moduli} we have explained that $\Modulin$ is
identified with an open subset of $\C^n$. Call it $U_n$ with
coordinates $s=(s_1,\cdots,s_n)$. Let $[x_0:\cdots:x_{n}]$ be the
homogenous coordinates of $\P^n$. Put
$$t_1=x_0,\cdots,t_{m-2}=x_{n},t_{m-1}=\sum_{i=0}^{n}x_i,t_m=x_0+\sum_{i=1}^ns_ix_i.$$
They give $m$ sections of $\sO(1)$ whose zero divisors in $\P^n$
meet transversally. Let $p_i, i=1,2$ be the projection of ${\rm
Tot}(L)\times U_n$ to the $i$-th factor. Define $$\sX_n\subset {\rm
Tot}(L)\times U_n$$ to be the zero locus of the section
$$
p_1^*s^r-\bigotimes_{i=1}^{m}(p\circ p_1)^*t_i\in \Gamma({\rm
Tot}(L)\times U_n,(p\circ p_1)^*\sO(m)),
$$
and $f_n: \sX_n\to U_n$ the composite $\sX_n\subset {\rm
Tot}(L)\times U_n\stackrel{p_2}{\to} U_n$. The is the family of
$r$-fold cyclic covers of $\P^n$ branched along a universal family
of $n+3$ hyperplane arrangements of $\P^n$ in general position. We
can do the simultaneous crepant resolution of the family $f_n$: The
divisors of ${\rm Tot}(L)\times U_n$ given by
$$
\tilde{E}_{0}:=(p_{1}^{*}s=0), \tilde{E}_{1}:=((p\circ
p_1)^*t_{1}=0),\cdots, \tilde{E}_{m}:=((p\circ p_1)^*t_{m}=0)
$$
meet transversally and the hypersurface $\mathcal{X}$ is binomial
with respect to
$$
(\mathfrak{E}=\{\tilde{E}_{0}\},
\mathfrak{F}=\{\tilde{E}_{1},\cdots, \tilde{E}_{m}\}).
$$
By Theorem \ref{resolution theorem}, we can apply the crepant
resolution algorithm to $\mathcal{X}$ and get a simultaneous crepant
resolution of $f_n$:
$$
\xymatrix{
  \tilde{\mathcal{X}} \ar[rr]^{ } \ar[dr]_{ }
                &  &   \sX_n \ar[dl]^{f_n}    \\
                & U_n                }
$$
This is not unique and a different choice leads to a fiberwise
birationally equivalent family of CY manifolds. We choose one and
call $\tilde f_n: \tilde \sX_n\to U_n=\Modulin$ \emph{the} family of
CY manifolds by our construction.\\
We can also consider $\Modulione$ and obtain a family $g_n: \sC\to
\Modulione$ whose fibers are $r$-fold cyclic covers of $\P^1$
branched along $n+3$ distinct points. The $n$-th self product
$(g_n)^n: \sC^n\to \Modulione$ of $g_n$ admits a natural action of
$G$ which acts fiberwisely as what we have described at the end of
\S\ref{cyclic cover of P^1}. For
$$
h_n: \sC^n/G\stackrel{\overline{(g_n)^n}}{\longrightarrow}
\Modulione\stackrel{\Gamma}{\longrightarrow}\Modulin,
$$
there is an isomorphism
$$
\xymatrix{
  \sC^n/G \ar[rr]^{\simeq} \ar[dr]_{h_n}
                &  & \sX_n    \ar[dl]^{f_n}    \\
                & \Modulin                 }
$$
Some notations before the main computational result: Write
$\V_{\Q}=R^1g_{n*}\Q$. By Lemma \ref{decomposition of weight one
HS}, the fiberwise cyclic automorphism induces a decomposition of
$\Q(\zeta_r)$-PVHS:
$$
\V_{\Q}\otimes \Q(\zeta_r)=\bigoplus_{i=1}^{r-1}\V_i.
$$
By Lemma \ref{decomposition of n-th wedge product}, there is a
decomposition of $\Q$-PVHSs:
$$
\bigwedge^n\V_\Q=\W_{1,\Q}\bigoplus \W_{2,\Q}
$$
such that $\W_{1,\Q}\otimes
\Q(\zeta_r)=\bigoplus_{i=1}^{r-1}\bigwedge^n\V_i$. Write
$\HH_n=R^n{\tilde f}_{n*}\Q$.
\begin{theorem}\label{structure of PVHS of the family}
Let $\tilde f_n: \tilde \sX_n\to \Modulin$ be the family of
$n$-dimensional CY manifolds. Then one has an isomorphism of
$\Q$-PVHSs:
$$
\HH_n\simeq \W_{1,\Q}.
$$
\end{theorem}
\begin{proof}
It follows from Propositions \ref{thm:Hodge structures under crepant
resolution} and \ref{isomorphism of PHS with W}.
\end{proof}
\begin{remark}
Note that, for and only for $n=3,5,9$, the $\Q(\zeta_r)$-PVHS
$$\bigwedge^n\V_1\bigoplus \bigwedge^n\V_{r-1}$$ underlies a sub
$\Q$-PVHS $\W_{unif,\Q}$ of $\W_{1,\Q}$. Write $\HH_{unif,\Q}$ to be
the sub $\Q$-PVHS of $\HH_n$ corresponding to $\W_{unif,\Q}$ under
the above isomorphism.
\end{remark}
\begin{corollary}
The following statements are true:
\begin{itemize}
    \item [(i)] The family $\tilde f_n$ is maximal.
    \item [(ii)] $\varsigma(\tilde f_n)=1$. Consequently, Shafarevich's
    conjecture holds true for $\tilde f_n$.
    \item [(iii)] A suitable partial compactification of the family $\tilde f_3$ is a Shimura family of $U(1,3)$-type.
    \item [(iv)] For $n=5,9$, the sub $\Q$-PVHS
    $\HH_{unif,\Q}\subset \HH_n$ gives a uniformization to a Zariski open subset of an arithmetic ball
    quotient.
\end{itemize}
\end{corollary}
\begin{proof}
Let $(F,\eta)$ be the corresponding Higgs bundle to $\V_\Q$,
$(E,\theta)$ to $\HH_n$. The above theorem on PVHS implies the
corresponding result on Higgs bundles, that is,
$$
(F,\eta)=\bigoplus_{i=1}^{r-1}(F_i,\eta_i),\quad (E,\theta)\simeq
\bigoplus_{i=1}^{r-1}\bigwedge^n(F_i,\theta_i).
$$
Lemma \ref{decomposition of weight one HS} implies that
$$
\rank(F_i^{1,0})=2i-1,\quad \rank(F_i^{0,1})=n+2-2i.
$$
Adding the information on the Hodge components, one has a more
explicit description of the Higgs bundle $(E,\theta)$ which is equal
to
$$
(E^{n,0}\oplus E^{n-1,1},\theta^{n,0})\bigoplus (E^{n-2,2}\oplus
E^{n-3,3},\theta^{n-2,2})\bigoplus \cdots\bigoplus (E^{1,n-1}\oplus
E^{0,n},\theta^{1,n-1}),
$$
together with isomorphisms of Higgs bundles
\begin{eqnarray*}
  (E^{n,0}\oplus E^{n-1,1},\theta^{n,0})&\simeq & (\wedge^{n}F^{1,0}_{r-1}\oplus
\wedge^{n-1}F_{r-1}^{1,0}\otimes
F_{r-1}^{0,1},\wedge^n\eta_{r-1}), \\
  &\cdots&   \\
    (E^{1,n-1}\oplus E^{0,n},\theta^{1,n-1})&\simeq& (\wedge^{n-1}F^{1,0}_1\otimes F^{0,1}_1\oplus
\wedge^nF^{0,1}_1,\wedge^n\eta_1).
\end{eqnarray*}
Note that the family $g_n: \sC\to \Modulione$ is in connection with
the theory of Deligne-Mostow \cite{DM} (see also \cite{Lo}). It
follows from Proposition 3.9 loc. cit. that
$$
\eta_{r-1}: F^{1,0}_{r-1}\to F^{0,1}_{r-1}\otimes
\Omega_{\Modulione}
$$
is an isomorphism. So is $\theta^{n,0}$. This shows (i). It follows
also that $\varsigma(\tilde f_n)=1$ as $$\theta^{n-1,1}:
E^{n-1,1}\to E^{n-2,2}\otimes \Omega_{\Modulin}$$ is simply zero. As
$\tilde f_n$ is maximal, the length of the Yukawa coupling of the
corresponding coarse moduli is also equal to one. By Theorem 6
\cite{LTYZ2}, Shafarevich's conjecture holds for this coarse moduli.
This shows (ii). The eigen-PVHS $\V_1$ is a special case of
$\C$-PVHSs over $\Modulione$ studied by Deligne-Mostow loc. cit.. In
our case,
$$\mu=(\mu_1=\frac{1}{r},\cdots,\mu_{m}=\frac{1}{r})$$ in the
notation of loc. cit.. It is shown in \S10 loc. cit. that when the
condition $\mathbf{INT}$ (see Theorem 3.11 loc. cit. for its
meaning) is satisfied, the monodromy representation of $\V_1$ has
discrete image in $\Aut(\B^n)$. The condition to guarantee the
discreteness of the image has been relaxed by Mostow \cite{Mostow0}
to $\mathbf{\sum INT}$ which turns out to be also sufficient by
Mostow \cite{Mostow}. In the appendix of \cite{Mostow}, he gives
also a complete list of $\mu$ satisfying $\mathbf{\sum INT}$ for
$n\geq 3$. Checking the list, one finds immediately that there are
exactly three cases labeled as No.1, No. 8 and No. 23 which have
equal $\mu_i$s, and they come exactly from $\V_1$ for $n=3,5,9$.
Moreover, in these three cases, the monodromy groups are actually
arithmetic. Note that
$$\HH_{unif,\Q}\otimes \Q(\zeta_r)\simeq (\V_1\bigoplus \V_{r-1})\otimes \Q(\zeta_r)(2-r),$$
and $\V_1, \V_{r-1}$ are dual to each other. The period map of
$\HH_{unif,\Q}$ gives an open embedding of $\Modulin$ into an
arithmetic ball quotient. For the case $n=3$, it is even more
special. In this case, one has
$$
\HH_3=\HH_{unif,\Q}\simeq \V_{\Q}\otimes \Q(-1).
$$
By Deligne-Mostow \cite{DM}, the family $g_3:\sC\to
\mathfrak{M}_{1,6}$ can be partially compactified so that the image
of its period map in the moduli space of principally polarized
Abelian 4-folds with a suitable level structure is an arithmetic
quotient of $\B^3$. See a recent work \cite{Moon} by B. Moonen on
the classification of Shimura subvarieties in the Jacobian locus of
moduli spaces of principally polarized abelian varieties arising
from cyclic covers of $\P^1$. This example appears as No. (10) in
Table 1 loc. cit.. It is quite obvious that the partial
compactification of $g_3$ yields the one of $f_3$ and also of
$\tilde f_3$. The detail is omitted since what it will involve is
not closely relevant to the paper. This completes the proof.
\end{proof}
We conclude the paper with the following remark.
\begin{remark}
(i). For $n=3$, the Hodge numbers of $\tilde X$ read
    $h^{1,1}=51,h^{2,1}=3$. J. Rohde has constructed in his doctor
    thesis (see \cite{Rohde}) a maximal family of CY 3-folds with
    the same Hodge numbers which is also a Shimura family. Note that
    the parameter of his family comes also from $\mathfrak{M}_{1,6}$. Are these two families birationally
    equivalent?\\
    (ii). One constructs more maximal families of CY manifolds from
    the moduli space of hyperplane arrangements in a projective
    space. Do our families exhaust all possibilities with length
    one Yukawa coupling?\\
    (iii). We have shown that $\HH_{unif,\Q}\subset \HH_n$ exists
    only for $n=3,5,9$, which are the unique three cases appeared in
    Mostow's list with equal $\mu_i$s. Is there some deeper reason than
a mere coincidence?
\end{remark}

\end{document}